\documentclass[12pt,reqno]{amsart}

\usepackage{amssymb}

\newcommand{\del}{\delta}
\newcommand{\eps}{\varepsilon}
\newcommand{\gam}{\gamma}
\newcommand{\lam}{\lambda}

\newcommand{\F}{{\mathbb F}}
\newcommand{\Z}{{\mathbb Z}}

\newcommand{\oC}{{\overline C}}

\newcommand{\cK}{{\mathcal K}}
\newcommand{\cR}{{\mathcal R}}

\newcommand{\lfl}{\left\lfloor}
\newcommand{\rfl}{\right\rfloor}
\newcommand{\lcl}{\left\lceil}
\newcommand{\rcl}{\right\rceil}

\newcommand{\stm}{\setminus}
\newcommand{\seq}{\subseteq}
\newcommand{\est}{\varnothing}

\newcommand{\longc}{,\ldots,}
\newcommand{\longp}{+\dotsb+}

\DeclareMathOperator{\rk}{rk}

\newcommand{\sub}[1]{_{\substack{#1}}}

\newtheorem{lemma}{Lemma}
\newtheorem{theorem}{Theorem}
\newtheorem{corollary}{Corollary}
\newtheorem{proposition}{Proposition}

\newcommand{\refl}[1]{~\ref{l:#1}}
\newcommand{\reft}[1]{~\ref{t:#1}}

\newcommand{\refp}[1]{~\ref{p:#1}}
\newcommand{\refb}[1]{~\cite{b:#1}}
\newcommand{\refe}[1]{~\eqref{e:#1}}

\theoremstyle{remark}
\newtheorem*{remark}{Remark}

\title[Progression-free sets in $\Z_4^n$]%
  {Progression-free sets in $\Z_4^n$ \\ are exponentially small}

\author{Ernie Croot}
\email{ecroot@math.gatech.edu}
\address{School of Mathematics, Georgia Institute of Technology, Atlanta,
  Georgia 30332, USA}

\author{Vsevolod F. Lev}
\email{seva@math.haifa.ac.il}
\address{Department of Mathematics, The University of Haifa at Oranim,
  Tivon 36006, Israel}

\author[P\'eter P\'al Pach]{P\'eter P\'al Pach$^\dag$}
\thanks{${}^\dag$ Supported by the Hungarian Scientific Research Funds
  (Grant Nr.~OTKA~PD115978 and OTKA~K108947) and the J\'anos Bolyai Research
  Scholarship of the Hungarian Academy of Sciences.}
\email{ppp@cs.bme.hu}
\address{Department of Computer Science and Information Theory, Budapest
  University of Technology and Economics, 1117 Budapest, Magyar tud\'osok
  k\"or\'utja 2, Hungary}

\begin{document}
\baselineskip=16pt

\begin{abstract}
We show that for integer $n\ge 1$, any subset $A\seq \Z_4^n$ free of
three-term arithmetic progressions has size $|A|\le 4^{\gam n}$, with an
absolute constant $\gam\approx 0.926$.
\end{abstract}

\maketitle

\section{Background and Motivation}\label{s:intro}

In his influential papers \cite{b:r1,b:r2}, Roth has shown that if a set
$A\seq\{1,2\longc N\}$ does not contain three elements in an arithmetic
progression, then $|A|=o(N)$ and indeed, $|A|=O(N/\log\log N)$ as $N$ grows.
Since then, estimating the largest possible size of such a set has become one
of the central problems in additive combinatorics. Roth's original results
were improved by Heath-Brown~\refb{h}, Szemer\'edi~\refb{sz},
Bourgain~\refb{b}, Sanders~\cite{b:s2,b:s3}, and Bloom~\refb{bl}, the current
record due to Bloom being $|A|=O(N(\log\log N)^4/\log N)$.

It is easily seen that Roth's problem is essentially equivalent to estimating
the largest possible size of a subset of the cyclic group $\Z_N$, free of
three-term arithmetic progressions. This makes it natural to investigate
other finite abelian groups.

We say that a subset $A$ of an (additively written) abelian group $G$ is
\emph{progression-free} if there do not exist pairwise distinct $a,b,c\in A$
with $a+b=2c$, and we denote by $r_3(G)$ the largest size of a
progression-free subset $A\seq G$. For abelian groups $G$ of odd order, Brown
and Buhler \refb{bb} and independently Frankl, Graham, and R\"odl \refb{fgr}
proved that $r_3(G)=o(|G|)$ as $|G|$ grows. Meshulam~\refb{m}, following the
general lines of Roth's argument, has shown that if $G$ is an abelian group
of odd order, then $r_3(G)\le 2|G|/\rk(G)$ (where we use the standard
notation $\rk(G)$ for the rank of $G$); in particular, $r_3(\Z_m^n)\le
2m^n/n$. Despite many efforts, no further progress was made for over 15
years, till Bateman and Katz in their ground-breaking paper \refb{bk} proved
that $r_3(\Z_3^n)=O(3^n/n^{1+\eps})$ with an absolute constant $\eps>0$.

Abelian groups of even order were first considered in \refb{l1} where, as a
further elaboration on the Roth-Meshulam proof, it is shown that
$r_3(G)<2|G|/\rk(2G)$ for any finite abelian group $G$; here $2G=\{2g\colon
g\in G\}$. For the homocyclic groups of exponent $4$ this result was improved
by Sanders~\refb{s3} who proved that $r_3(\Z_4^n)=O(4^n/n(\log n)^\eps)$ with
an absolute constant $\eps>0$. The goal of this paper is to further improve
Sanders's result, as follows.

Let $H$ denote the binary entropy function; that is,
  $$ H(x) = -x\log_2x - (1-x)\log_2(1-x),\quad x\in(0,1), $$
where $\log_2x$ is the base-$2$ logarithm of $x$.
\begin{theorem}\label{t:main}
If $n\ge 1$ and $A\seq\Z_4^n$ is progression-free, then letting
  $$ \gam := \max \Big\{ \frac12 \big( H(0.5-\eps)+H(2\eps) \big)
                                  \colon 0<\eps<0.25 \Big\} \approx 0.926 $$
we have
  $$ |A| \le 4^{\gam n}. $$
\end{theorem}
The proof of Theorem~\reft{main} is presented in the next section.

We note that the exponential reduction in Theorem~\reft{main} is first of a
kind for problems of this sort.

Starting from Roth, the standard way to obtain quantitative estimates for
$r_3(G)$ involves a combination of the Fourier analysis and the density
increment technique; the only exception is \refb{l2} where for the groups
$G\cong\Z_q^n$ with a prime power $q$, the above-mentioned Meshulam's result
is recovered using a completely elementary argument. In contrast, in the
present paper we use the polynomial method, without resorting to the familiar
Fourier analysis -- density increment strategy.

For a finite abelian group $G\cong\Z_{m_1}\oplus\dotsb\oplus\Z_{m_k}$ with
positive integer $m_1\mid\dotsb\mid m_k$, denote by $\rk_4(G)$ the number of
indices $i\in[1,k]$ with $4\mid m_i$. Since, writing $n:=\rk_4(G)$, the group
$G$ is a union of $4^{-n}|G|$ cosets of a subgroup isomorphic to $\Z_4^n$, as
a direct consequence of Theorem~\reft{main} we get the following corollary.
\begin{corollary}
If $G$ is a finite abelian group then, writing $n:=\rk_4(G)$, we have
$r_3(G)\le 4^{-(1-\gam)n}|G|$, where $\gam\approx 0.926$ is the constant of
Theorem~\reft{main}.
\end{corollary}

\section{Proof of Theorem~\reft{main}}\label{s:proof}

The proof of Theorem~\reft{main} is based on the following lemma.

\begin{lemma}\label{l:main}
Suppose that $n\ge 1$ and $d\ge 0$ are integers, $P$ is a multilinear
polynomial in $n$ variables of total degree at most $d$ over a field
$\F$, and $A\seq\F^n$ is a set with $|A|>2\sum_{0\le i\le d/2}\binom ni$.
If $P(a-b)=0$ for all $a,b\in A$ with $a\ne b$, then also $P(0)=0$.
\end{lemma}

\begin{proof}
Let $m:=\sum_{0\le i\le d/2}\binom ni$, and let $\cK=\{K_1\longc K_m\}$
be the collection of all sets $K\seq[n]$ with $|K|\le d/2$. Writing for
brevity
  $$ x^I := \prod_{i\in I} x_i,\quad x=(x_1\longc x_n)\in\F^n,\ I\seq[n], $$
there exist coefficients $C_{I,J}\in\F\ (I,J\seq[n])$ depending only on the
polynomial $P$, such that for all $x,y\in\F^n$ we have
\begin{align*}\label{e:cIJ}
  P(x-y)
    &= \sum\sub{I,J\seq[n] \\ I\cap J=\est \\ |I|+|J|\le d}
                                                 C_{I,J}\, x^I y^J\notag \\
    &= \sum_{I\in\cK} x^I \sum\sub{J\seq[n]\stm I \\ |J|\le d-|I|}
                                                              C_{I,J}\, y^J
       + \sum_{J\in\cK} \Bigg( \sum\sub{I\seq[n]\stm J \\ d/2<|I|\le d-|J|}
                                                  C_{I,J}\, x^I \Bigg) y^J.
\end{align*}
The right-hand side can be interpreted as the scalar product of the
vectors $u(x),v(y)\in\F^{2m}$ defined by
\vskip -0.25in
\begin{gather*}
  u_i(x)     = x^{K_i}, \quad
  u_{m+i}(x) = \sum\sub{I\seq[n]\stm K_i \\ d/2<|I|\le d-|K_i|}
                           C_{I,K_i}\, x^I
  \intertext{and}
  v_i(y) = \sum\sub{J\seq[n]\stm K_i \\ |J|\le d-|K_i|} C_{K_i,J}\, y^J,
  \quad v_{m+i}(y) = y^{K_i}
\end{gather*}
for all $1\le i\le m$. Consequently, if we had $P(a-b)=0$ for all
 $a,b\in A$ with $a\ne b$, while $P(0)\ne 0$, this would imply that the
vectors $u(a)$ and $v(b)$ are orthogonal if and only if $a\ne b$. As a
result, the vectors $u(a)$ would be linearly independent (an equality of the
sort $\sum_{a\in A}\lam_au(a)=0$ with the coefficients $\lam_a\in\F$ after a
scalar multiplication by $v(b)$ yields $\lam_b=0$, for any $b\in A$).
Finally, the linear independence of
 $\{u(a)\colon a\in A\}\seq\F^{2m}$ implies $|A|\le 2m$, contrary to
the assumptions of the lemma.
\end{proof}

\begin{remark}
It is easy to extend the lemma relaxing the multilinearity assumption to the
assumption that $P$ has bounded degree in each individual variable.
Specifically, denoting by $f_\del(n,d)$ the number of monomials
$x_1^{i_1}\dotsc x_n^{i_n}$ with $0\le i_1\longc i_n\le\del$ and
 $i_1\longp i_n\le d$, if $P$ has all individual degrees not exceeding $\del$,
and the total degree not exceeding $d$, then $|A|>2f_\del(n,\lfl d/2\rfl)$
along with $P(a-b)=0\ (a,b\in A,\ a\ne b)$ imply $P(0)=0$. Moreover, taking
$\del=d$, or $\del=|\F|-1$ for $\F$ finite, one can drop the individual
degree assumption altogether.
\end{remark}

We will use the estimate
\begin{equation}\label{e:binom}
  \sum_{0\le i\le z} \binom ni < 2^{nH(z/n)}
\end{equation}
valid for all integer $n\ge 1$ and real $0<z\le n/2$.

Recall, that for integer $n\ge d\ge 0$, the sum $\sum_{i=0}^d\binom ni$ is
the dimension of the vector space of all multilinear polynomials in $n$
variables of total degree at most $d$ over the two-element field $\F_2$. In
particular, the dimension of the vector space of \emph{all} multilinear
polynomials in $n$ variables over $\F_2$ is equal to the dimension of the
vector space of all $\F_2$-valued functions on $\F_2^n$, and it follows that
any non-zero multilinear polynomial represents a non-zero function. These
basic facts are used in the proof of Proposition~\refp{richcosets} below.

For integer $n\ge 1$, denote by $F_n$ the subgroup of the group $\Z_4^n$
generated by its involutions; thus, $F_n$ is both the image and the kernel of
the doubling endomorphism of $\Z_4^n$ defined by $g\mapsto 2g\ (g\in\Z_4^n)$,
and we have $F_n\cong\Z_2^n$.

\begin{proposition}\label{p:richcosets}
Suppose that $n\ge 1$ and $A\seq\Z_4^n$ is progression-free. Then for every
$0<\eps<0.25$, the number of $F_n$-cosets containing at least
$2^{nH(0.5-\eps)+1}$ elements of $A$ is less than $2^{nH(2\eps)}$.
\end{proposition}

\begin{proof}
Let $\cR$ be the set of all those $F_n$-cosets containing at least
$2^{nH(0.5-\eps)+1}$ elements of $A$, and for each coset $R\in\cR$ let
$A_R:=A\cap R$; thus, $\cup_{R\in\cR} A_R \seq A$ (where the union is
disjoint), and
\begin{equation}\label{e:Aslarge}
  |A_R| \ge 2^{nH(0.5-\eps)+1},\quad R\in\cR.
\end{equation}

For a subset $S\seq\Z_4^n$, write
  $$ 2\cdot S:=\{s'+s''\colon (s',s'')\in S\times S,\ s'\ne s'' \}
       \quad\text{and}\quad 2\ast S:=\{2s\colon s\in S \}. $$
The assumption that $A$ is progression-free implies that the sets
  $$ B := \cup_{R\in\cR}(2\cdot A_R)\seq F_n\quad\text{and}\quad
                                        C:=\cup_{R\in\cR}(2\ast R)\seq F_n $$
are disjoint: this follows by observing that if $2r\in 2\cdot A_R$ with some
$r\in R$, then for each $a\in r+F_n$ we have $2a=2r\in 2\cdot A_R\seq 2\cdot
A$. Furthermore, the sets $2\ast R$ are in fact pairwise distinct singletons
(for $2r_1=2r_2$ is equivalent to $r_1-r_2\in F_n$ and thus to
$r_1+F_n=r_2+F_n$), whence $|C|=|\cR|$.

Let $d=n-\lcl 2\eps n\rcl$ so that, in view of \refe{Aslarge} and
\refe{binom},
\begin{equation}\label{e:was1}
  2\sum_{0\le i\le d/2} \binom ni < 2^{nH(0.5-\eps)+1} \le |A_R|, \quad R\in\cR.
\end{equation}
Denoting by $\overline C$ the complement of $C$ in $F_n$, and assuming that
$|\cR|\ge 2^{nH(2\eps)}$ (contrary to what we want to prove),
from~\refe{binom} we get
  $$ \sum_{i=0}^d \binom ni = 2^n - \sum_{i=0}^{\lcl 2\eps n\rcl-1} \binom ni
      > 2^n - 2^{nH(2\eps)} \ge 2^n - |\cR| = 2^n - |C| = |\oC|. $$
Consequently, identifying $F_n$ with the additive group of the vector space
$\F_2^n$, and accordingly considering $B$ and $C$ as subsets of $\F_2^n$, we
conclude that the dimension of the vector space of all multilinear
$n$-variate polynomials over the field $\F_2$ exceeds the dimension of the
vector space of all $\F_2$-valued functions on $\oC$. Thus, the evaluation
map, associating with every polynomial the corresponding function, is
degenerate. As a result, there exists a non-zero multilinear polynomial
$P\in\F_2[x_1\longc x_n]$ of total degree $\deg P\le d$ such that $P$
vanishes on $\overline C$. In particular, $P$ vanishes on $B\seq\overline C$,
and therefore on each set $2\cdot A_R$, for all $R\in\cR$. Fixing arbitrarily
an element $r\in R$, the polynomial $P(2r+x)$ thus vanishes whenever $x\in
2\cdot(A_R-r)$. Hence, also $P(2r)=0$ by Lemma~\refl{main} (which is
applicable in view of \refe{was1}); that is, $P$ also vanishes on each
singleton set $2\ast A_R$, for all $R\in\cR$. It follows that $P$ vanishes on
$C$. However, $P$ was chosen to vanish on $\overline C$. Therefore, $P$
vanishes on all of $\F_2^n$, and it follows that $P$ is the zero polynomial.
This is a contradiction showing that $|\cR|<2^{nH(2\eps)}$, and thus
completing the proof.
\end{proof}

\begin{proof}[Proof of Theorem~\reft{main}]
For $x\ge 0$, let $N(x)$ denote the number of $F_n$-cosets containing at
least $x$ elements of $A$; thus $N(x)=0$ for $x>2^n$, and we can write
\begin{equation}\label{e:int1}
  |A| = \int_0^{2^{n+1}} N(x)\,dx.
\end{equation}
Trivially, we have $N(x)\le2^n$ for all $x\ge 0$, so that
\begin{equation}\label{e:int2}
  \int_0^{2^{nH(1/4)+1}} N(x)\,dx \le 2^{(H(1/4)+1)n+1} < 2\cdot 4^{\gam n}.
\end{equation}
On the other hand, the substitution $x=2^{nH(0.5-\eps)+1}$ gives
\begin{equation}\label{e:int3}
  \int_{2^{nH(1/4)+1}}^{2^{n+1}} N(x)\,dx
      = n \int_0^{1/4} 2^{nH(0.5-\eps)+1} N(2^{nH(0.5-\eps)+1})
                          \,\log\frac{0.5+\eps}{0.5-\eps} \,d\eps,
\end{equation}
and applying Proposition~\refp{richcosets}, the integral in the right-hand
side can be estimated as
\begin{equation}\label{e:int4}
  2n \int_0^{1/4} 2^{n(H(0.5-\eps)+H(2\eps))}
                       \,\log\frac{0.5+\eps}{0.5-\eps} \,d\eps
            < 3n \int_0^{1/4} 2^{n(H(0.5-\eps)+H(2\eps))} \, d\eps
                                                        < n\cdot 2^{\gam n}.
\end{equation}
From \refe{int1}--\!\!\refe{int4} we get $|A|<(n+2)\cdot4^{\gam n}$, and to
conclude the proof we use the tensor power trick: for integer $k\ge 1$, the
set $A\times\dotsb\times A\seq\Z_4^{kn}$ is progression-free and therefore
  $$ |A|^k < (kn+2)\cdot 4^{\gam kn} $$
by what we have just shown. This readily implies the result.
\end{proof}

\bigskip

\end{document}